\documentclass[11pt]{amsart}

\usepackage{amsfonts,amsmath,amssymb,amsthm,enumerate}
\usepackage[latin1]{inputenc}
\usepackage{hyperref}
\usepackage{tikz}
\usetikzlibrary{calc}
\usepackage[all]{xy}
\usepackage{float}

\newcommand{\Rr}{\mathbb{R}}

\newcommand{\Zz}{\mathbb{Z}}

\renewcommand {\le}{\leqslant}
\renewcommand {\ge}{\geqslant}

\newcommand{\defi}[1]{\emph{#1}}
\newcommand{\lk}{\mathop{\mathrm{lk}}\nolimits} 
\newcommand{\Conf}{\mathop{\mathrm{Conf}}\nolimits} 
\newcommand{\Aut}{\mathop{\mathrm{Aut}}\nolimits} 
\newcommand{\Aff}{\mathop{\mathrm{Aff}}\nolimits} 
\newcommand{\conjug}[1]{\bar{#1}} 
\newcommand{\id}{\mathop{\mathrm{id}}\nolimits} 
\renewcommand{\Im}{\mathop{\mathrm{Im}}\nolimits} 
\newcommand{\Ker}{\mathop{\mathrm{Ker}}\nolimits}

{\theoremstyle{plain}
\newtheorem{theorem}{Theorem}[section]    
\newtheorem*{theorem*}{Theorem}

\newtheorem{lemma}[theorem]{Lemma}       
\newtheorem{proposition}[theorem]{Proposition}      
\newtheorem{proposition*}{Proposition} 
\newtheorem{corollary}[theorem]{Corollary}      

}
{\theoremstyle{remark}

\newtheorem*{remark*}{Remark}  
\newtheorem{remark}[theorem]{Remark}   

\newtheorem*{question*}{Question}
}

\title{The braid group of a necklace}

\date{\today}

\author{Paolo Bellingeri}
\email{paolo.bellingeri@unicaen.fr}

\author{Arnaud Bodin}
\email{Arnaud.Bodin@math.univ-lille1.fr}

\address{Laboratoire Nicolas Oresme,
 Universit\'e de Caen, 14032 Caen, France}  

\address{Laboratoire Paul Painlev\'e, Math\'ematiques, Universit\'e 
Lille 1, 59655 Villeneuve d'Ascq, France}



\begin{document}

\begin{abstract}
We show several geometric and algebraic aspects of a \emph{necklace}:
a link composed with a core circle and a series of (unlinked) circles linked to this core.
We first prove that the fundamental group of the configuration space of necklaces 
(that we will call  braid group of a necklace) is 
isomorphic to  the braid group over an annulus quotiented by the square of the center.
We then define   braid groups of necklaces and  affine braid groups of type $\mathcal{A}$
in terms  of  automorphisms of  free groups and characterize these automorphisms 
among all automorphisms of free groups.  In the case of affine braid groups of type $\mathcal{A}$
such a representation is faithful.
\end{abstract}

\maketitle

\section{Introduction}

The braid group $B_n$ can be defined as the fundamental group of the configuration space  
of $n$ distinct points on a disk.
By extension, when we replace the disk by another surface $\Sigma$, we define the \emph{braid group on $n$ strands over $\Sigma$}
as the fundamental group of the configuration space of $n$ distinct points on $\Sigma$.
A particular case is when $\Sigma$ is the annulus: the braid group of the annulus,
denoted by $CB_n$ (for \emph{circular braid group}) is the fundamental group of the configuration 
space $\mathcal{CB}_n$ of $n$ distinct points over an annulus.
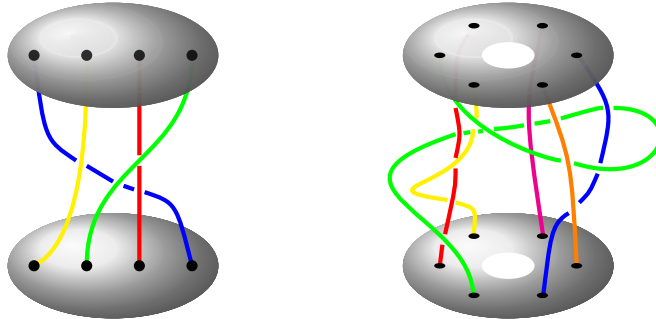
\begin{figure}
 \begin{tikzpicture}[scale=0.7]
\fill [ball color=gray!20] (0,0) ellipse (2 and 1);

\draw[ultra thick,yellow] (-1.5,0) .. controls (-0.5,0.5)  and (-0.5,3) .. (-0.5,4);

\draw[ultra thick,green] (-0.5,0) .. controls (-0.5,2)  and (1.5,2) .. (1.5,4);

\draw[ultra thick,red] (0.5,0)--(0.5,1.9);
\draw[ultra thick,red] (0.5,2.1)--(0.5,4);

\draw[ultra thick,blue] (1.5,0) .. controls (1.2,1.2)   ..(0.6,1.4);
\draw[ultra thick,blue] (0.4,1.45)--(0.2,1.52);
\draw[ultra thick,blue] (0,1.6)--(-0.5,1.9);
\draw[ultra thick,blue] (-0.7,2).. controls (-1.3,2.4)   ..(-1.5,4);

\fill [ball color=gray!20,opacity=0.6] (0,4) ellipse (2 and 1);

\foreach \i in{-1.5,-0.5,0.5,1.5}{
  \fill[black,opacity=1]({\i},0) circle(3pt);
  \fill[black,opacity=0.8]({\i},4) circle(3pt);
};

\end{tikzpicture} \qquad
 \begin{tikzpicture}[scale=0.7]
\fill [ball color=gray!20] (0,0) ellipse (2 and 1);
\fill [white] (0,0) ellipse (0.5 and 0.25);

\draw[ultra thick,orange] (0:1.3)  .. controls (1.2,2).. (0.66,3.4);

\draw[ultra thick,blue] (0.66,-0.6).. controls (0.8,0.5)  and (0.8,0.8) ..(1.15,1);
\draw[ultra thick,blue] (1.35,1.20).. controls (1.55,1.4)..(1.7,1.75);
\draw[ultra thick,blue] (1.75,1.95).. controls (2,3)   ..(1.33,4);

\draw[ultra thick,magenta] (0.66,0.6) .. controls (0.3,3)  .. (0.66,4.6);

\draw[ultra thick,yellow] (-0.66,0.6) .. controls (-0.6,1.1)  .. (-1,1.2);
\draw[ultra thick,yellow] (-1.15,1.25) .. controls (-2.9,1.5) and (-0.7,1.5) .. (-0.66,2.6);
\draw[ultra thick,yellow] (-0.6,2.8) .. controls (-0.6,3)  .. (-0.66,3.4);

\draw[ultra thick,green] (-0.66,-0.6) .. controls (-0.66,1)  and (-4,1.5)  .. (-1,2.5);
\draw[ultra thick,green] (-0.85,2.52) -- (-0.72,2.54);
\draw[ultra thick,green] (-0.57,2.55) --(-0.48,2.56);
\draw[ultra thick,green] (-0.27,2.6).. controls (0,2.62) .. (0.3,2.67);
\draw[ultra thick,green] (0.5,2.7) -- (0.8,2.75);
\draw[ultra thick,green] (1,2.8) .. controls (1.7,3) .. (1.8,3);
\draw[ultra thick,green] (2,3) .. controls (3.5,3) and (3,1.4) .. (1.25,1.95);
\draw[ultra thick,green] (1,2) .. controls (0.8,2)  and (-1.3,2.6)  .. (-1.3,4);

\draw[ultra thick,red] (-1.3,0) -- (-1.25,0.4);
\draw[ultra thick,red] (-1.2,0.6) .. controls (-1,1.5) .. (-1,1.85);
\draw[ultra thick,red] (-1,2.05) .. controls (-0.9,2.5) .. (-1,3);
\draw[ultra thick,red] (-1,3.2) .. controls (-1,4)  .. (-0.66,4.6);

\fill [ball color=gray!20,opacity=0.6] (0,4) ellipse (2 and 1);
\fill [white] (0,4) ellipse (0.5 and 0.25);

\begin{scope}[yscale=0.5]
\foreach \i in{0,60,...,300}{
  \fill[black,opacity=1]({\i}:1.3) circle(3pt);
  \fill[black,opacity=1] (0,8)+({\i}:1.3) circle(3pt);
};
\end{scope}

\end{tikzpicture}
 \caption{A braid with $4$ strands over a disk (left). 
 A braid with $6$ strands over an annulus (right).}
\end{figure}


There is a 3-dimensional analogue of $B_n$: it is the fundamental group of all configurations of $n$
unlinked Euclidean circles. Following  \cite{BH} we will denote by $\mathcal{R}_n$ the space of configurations of $n$
unlinked Euclidean circles and by $R_n$ its fundamental group (called \emph{group of rings} in \cite{BH}).
The group $R_n$  is generated by $3$ types of moves (see figure \ref{fig:moves}).
\begin{figure}
 \begin{tikzpicture}[scale=0.6]

 \draw[ultra thick]  (0:1) arc (0:360:1);

\begin{scope}[xshift=3cm]
  \draw[ultra thick]  (-165:1) arc (-165:0:1);
  \draw[ultra thick]  (-180:1) arc (-180:-345:1);
\end{scope}

\draw[->,>=latex,ultra thick,blue] (1.2,0.) .. controls (2,-0.5)  and (2.5,0.35) .. (5,0.05);

 \node at (0,-1.5) {$i$}; 
 \node at (3,-1.5) {$i+1$}; 
\end{tikzpicture} \qquad
 \begin{tikzpicture}[scale=0.6]

 \draw[ultra thick]  (0:1) arc (0:360:1);

\begin{scope}[xshift=3cm]
  \draw[ultra thick]  (-190:1) arc (-190:155:1);
\end{scope}

\draw[ultra thick,blue] (1.2,0.1) .. controls (2,0.8)  and (2.5,-0.3) .. (3.8,0.1);
\draw[->,>=latex,ultra thick,blue] (4.15,0.15) .. controls (4.3,0.25) .. (5,0.45);

 \node at (0,-1.5) {$i$}; 
 \node at (3,-1.5) {$i+1$}; 
\end{tikzpicture} \qquad
 \begin{tikzpicture}[scale=0.6]
\draw[gray] (0,-1.5)--(0,2);

 \draw[ultra thick]  (0:1) arc (0:360:1);
\draw[->,ultra thick,blue]  (0.4,2.5) arc (90:-90:0.8 and 0.4) ;

 \node at (0.5,-1.5) {$i$}; 
 \node at (1.9,2.1) {$180^ \circ$};
\end{tikzpicture}
 \caption{The move $\rho_i$ (left).
 The move $\sigma_i$ (center).
 The move $\tau_i$ (right).
 \label{fig:moves}}
\end{figure}
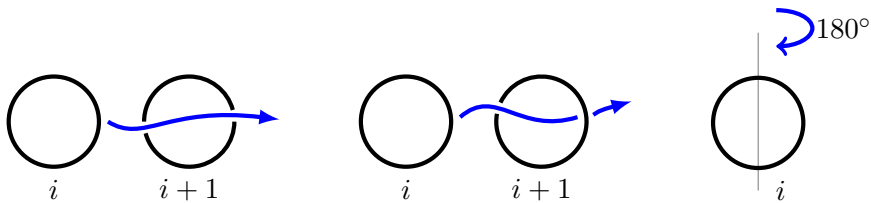

The move $\rho_i$ is the path permuting the $i$-th and the $i+1$-th circles by passing over (or around)
while $\sigma_i$ permutes them by passing the $i$-th circle through the $i+1$-th and $\tau_i$ is a $180^\circ$ rotation of the circle back to itself, 
which satisfies $\tau_i^2 = \id$ (let us remark that  our notation is opposite to the one of  \cite{BH}, where  $\rho_i$ was denoted by $\sigma_i$ and $\sigma_i$ by $\rho_i$; here we change the notation for the simple reason that 
 $\sigma_i$'s generate a subgroup isomorphic to $B_n$).
To avoid the last move $\tau_i$   one can define
$\mathcal{UR}_n$ as the configuration of $n$ unlinked Euclidean circles being all parallel to a fixed plane, say the $yz$-plane
(\emph{untwisted rings}).
The fundamental group of this configuration space is denoted in \cite{BH} by $UR_n$ but we will denote it by $WB_n$
and we shall call it  the  \defi{welded braid group}, since this is the most usual name for  this group  which appears in the literature in other  different contexts such as
motion groups (\cite{G1},\cite{G2}), ribbon tubes (\cite{ABMW})
automorphisms of free groups (\cite{BP,FRR}) 
or collections of paths (\cite{FRR}).

This article is devoted to the relationship between   configuration spaces of points 
over an annulus and configuration spaces of Euclidean circles. We will consider a special type of link, 
composed of one core and $n$ Euclidean circles linked to this core, the $n$ circles being the boundary of disks
that intersect the core circle orthogonally at the center of the disks.  We will call 
such a link a \defi{necklace} on $n$ components and  we will denote by $\mathcal{L}_n$. 
We will introduce the configuration space of  necklaces  and we will construct a representation
of its fundamental group  as a subgroup of automorphisms of free groups.
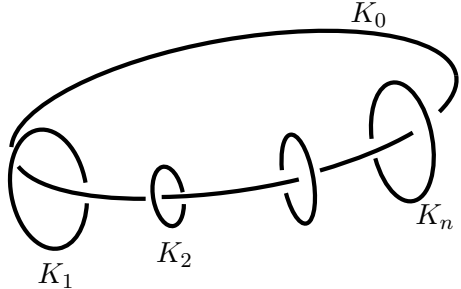
\begin{figure}
 \begin{tikzpicture}[scale=1]

\begin{scope}[rotate=10]


\draw[ultra thick] (3,0) arc (0:175:3 and 1);
\draw[ultra thick] (3,0) arc (0:-26:3 and 1);
\draw[ultra thick] (2.3,-0.65) arc (-42:-72:3 and 1) ;
\draw[ultra thick] (0.7,-1) arc (-80:-116:3 and 1) ;
\draw[ultra thick]  (-1.35,-0.9) arc (-117:-170:3 and 1) ;


 \node at (2,1) {$K_0$}; 

 \node at (-2.7,-1.7) {$K_1$}; 
\draw[ultra thick] (-2.1,-0.55) arc (0:340:0.5 and 0.8);

 \node at (-1.1,-1.7) {$K_2$}; 
\draw[ultra thick] (-0.85,-0.90) arc (3:340:0.2 and 0.4);

\draw[ultra thick] (0.5,-1.15) arc (-165:165:0.2 and 0.6);

 \node at (2.4,-1.8) {$K_n$}; 
\draw[ultra thick] (1.75,-0.9) arc (-169:172:0.4 and 0.8);

\end{scope}

\end{tikzpicture} 
 \caption{The necklace $\mathcal{L}_n$.\label{fig:necklace}}
\end{figure}

More precisely, our first result (Theorem \ref{th:circular}) is that the fundamental group of the configuration space of $n$-component necklaces (that we will call the  \emph{braid group of a necklace})
is isomorphic to 
the fundamental group of the configuration space of $n$ points 
over an annulus (which is $CB_n$)  quotiented by the square of its center.  

A theorem of Artin characterizes automorphisms of a free group coming 
from the action of the standard braid group. In our case, to a loop $\mathcal{L}_n(t)$ of necklaces we associate
an automorphism from $\pi_1(\Rr^3\setminus\mathcal{L}_n)$ to itself.
Our second result, Theorem \ref{th:circularartin},  is an analogue of Artin's theorem for the braid group of a necklace.

In section \ref{sec:zero}, we define affine braid groups of type $\mathcal{A}$ in terms of configurations of Euclidean circles
and we refine the representation given  in Theorem \ref{th:circularartin} to obtain
 a faithful representation and a characterization as automorphisms of free groups for affine braid groups of type $\mathcal{A}$ (Theorem \ref{thm:affine}):
 this is the third main result of the paper.
In section \ref{ssec:linearnecklace} we show how to define the braid group $B_n$
 in terms of configurations of Euclidean circles and we give a short survey on some remarkable (pure) subgroups
 of $WB_n$; finally  in  Section \ref{2.1proof} we determine  the kernel of a particular representation of $CB_n$ in $\Aut F_n$, proving this way the statement of Theorem
 \ref{thmCB}, which plays a key role in the proof of  Theorem \ref{thm:affine}.

\section{Necklaces and circular braids}
\label{sec:necklace}


\subsection{The circular braid group}

Recall that the circular braid group $CB_n$ is the fundamental group
of $n$ distinct points in the plane less the origin (i.e.~topologically an annulus).
The circular braid group $CB_n$
admits the following  presentation (where the indices are defined modulo $n$, see for instance \cite{KP}):
$$CB_n = \left\langle \sigma_1,\ldots,\sigma_n, \zeta \mid
\begin{array}{l}
\sigma_i\sigma_{i+1}\sigma_i = \sigma_{i+1}\sigma_i\sigma_{i+1} \quad \text{ for } i = 1,2\ldots, n, \\
\sigma_i\sigma_j=\sigma_j\sigma_i \quad \text{ for } |i-j| \neq 1, \\
\conjug{\zeta} \sigma_i \zeta =\sigma_{i+1} \quad \text{ for } i = 1,2\ldots, n \\
\end{array}
\right\rangle.
$$
where $\conjug{\zeta}$ stands for $\zeta^{-1}$.
Geometrically $\sigma_i$ permutes the $i$-th and $i+1$-th point and $\zeta$ is cyclic permutation
of the points.

\begin{figure}
 \begin{tikzpicture}[scale=0.7]
\fill [ball color=gray!20] (0,0) ellipse (2 and 1);
\fill [white] (0,0) ellipse (0.5 and 0.25);

\draw[ultra thick,blue]  (-0.66,-0.6)  .. controls (0.5,1.5).. (0.66,3.4);

\draw[ultra thick,red] (0.66,-0.6) .. controls (0.6,0.5)  .. (0.25,0.8);
\draw[ultra thick,red] (0.05,1) .. controls (-0.5,1.5)  .. (-0.66,3.4);

\draw[ultra thick,green] (-1.33,0).. controls (-1.8,1)  and (-1.8,3)   ..(-1.33,4);

\draw[ultra thick,green] (1.33,0).. controls (1.8,1)  and (1.8,3)   ..(1.33,4);

\draw[ultra thick,green] (0.66,0.6) .. controls (1,2.5)  .. (0.66,4.6);

\draw[ultra thick,green] (-0.66,0.6)  .. controls (-1,2.5)  .. (-0.66,4.6);

\fill [ball color=gray!20,opacity=0.6] (0,4) ellipse (2 and 1);
\fill [white] (0,4) ellipse (0.5 and 0.25);

\begin{scope}[yscale=0.5]
\foreach \i in{0,60,...,300}{
  \fill[black,opacity=1]({\i}:1.3) circle(3pt);
  \fill[black,opacity=1] (0,8)+({\i}:1.3) circle(3pt);
};
\end{scope}

\end{tikzpicture} \qquad\qquad
 \begin{tikzpicture}[scale=0.7]
\fill [ball color=gray!20] (0,0) ellipse (2 and 1);
\fill [white] (0,0) ellipse (0.5 and 0.25);

\draw[ultra thick,blue]  (0.66,-0.6)  -- (-0.66,3.4);

\draw[ultra thick,blue]  (-0.66,-0.6)  -- (-1.33,4);

\draw[ultra thick,blue]  (1.33,0)  -- (0.66,3.4);

\draw[ultra thick,blue]  (0.66,0.6)  -- (0.9,1.8);
\draw[ultra thick,blue]  (1,2.1)  -- (1.33,4);

\draw[ultra thick,blue]  (-0.64,0.6)  -- (-0.25,1.8);
\draw[ultra thick,blue]  (-0.1,2.1)  -- (0.66,4.6);

\draw[ultra thick,blue]  (-1.33,0)  -- (-1.1,1.8);
\draw[ultra thick,blue]  (-1,2.25)  -- (-0.66,4.6);

\fill [ball color=gray!20,opacity=0.6] (0,4) ellipse (2 and 1);
\fill [white] (0,4) ellipse (0.5 and 0.25);

\begin{scope}[yscale=0.5]
\foreach \i in{0,60,...,300}{
  \fill[black,opacity=1]({\i}:1.3) circle(3pt);
  \fill[black,opacity=1] (0,8)+({\i}:1.3) circle(3pt);
};
\end{scope}

\end{tikzpicture}
 \caption{A move $\sigma_i$ (left). 
 The move $\zeta$ (right).}
\end{figure}
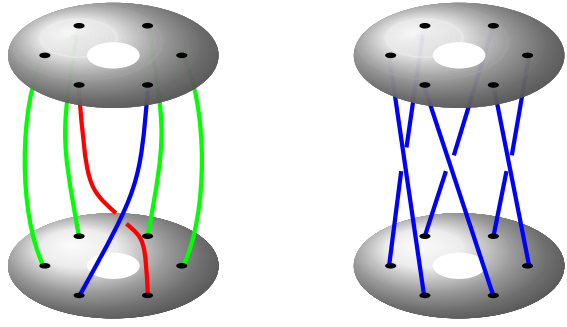

This is close to a presentation of the classical braid group, with two major differences:
(a) the indices are defined modulo $n$; (b) there are additional relations $\conjug{\zeta} \sigma_i \zeta =\sigma_{i+1}$.
In fact these latter relations make it possible to generate $CB_n$
with only two generators: $\sigma_1$ and $\zeta$.
We will consider the representation $\rho_{CB}:  CB_n \to \Aut F_n$   defined as follows: 
$$\rho_{CB}(\sigma_i) : 
\left\{\begin{array}{l}
x_i \mapsto x_{i} x_{i+1} \conjug{x_{i}} \\
x_{i+1} \mapsto x_i \\
x_{j} \mapsto x_j \quad j \neq i,i+1\\
\end{array}\right.
\qquad 
\rho_{CB}(\zeta) : 
\left\{\begin{array}{l}
x_j \mapsto x_{j+1} \\      
\end{array}\right.$$
where indices are modulo $n$.
The following Theorem will be proved in the last section and will play a key role in next sections:
we postpone the proof since it needs some reminders on Artin representation of $B_n$ in $\Aut F_n$
and since the techniques are slightly different from the rest of the paper.

\begin{theorem} \label{thmCB}
The kernel of $\rho_{CB}:  CB_n \to \Aut F_n$ is the cyclic group generated by $\zeta^n$.
\end{theorem}


\subsection{Necklaces}

A link  $\mathcal{L}_n = K_0 \cup K_1 \cup  \ldots\cup K_n$  is called a \defi{necklace} if:
\begin{itemize}
  \item $K_0$ is an Euclidean circle of center $O$ and of radius $1$,
  
  \item each Euclidean circle $K_i$, $i>0$, has a center $O_i$ belonging to $K_0$,
  the circle $K_i$ being of radius $r_i$ with $0 < r_i < \frac12$
  and belonging to the plane containing the line $(OO_i)$ and perpendicular to 
  the plane of $K_0$,
  
  \item if $O_i=O_j$ then $r_i \neq r_j$.
\end{itemize}

Here is an example of necklace:
\begin{figure}
 \begin{tikzpicture}[scale=1]

\begin{scope}[rotate=-10]


\draw[ultra thick] (3,0) arc (0:175:3 and 1);
\draw[ultra thick] (3,0) arc (0:-26:3 and 1);
\draw[ultra thick] (2.3,-0.65) arc (-42:-72:3 and 1) ;
\draw[ultra thick] (0.7,-1) arc (-80:-116:3 and 1) ;
\draw[ultra thick]  (-1.35,-0.9) arc (-117:-170:3 and 1) ;


 \node at (-2.5,1) {$K_0$}; 
 \fill[blue] (0,0) circle (2pt);
 \node[above] at (0,0) {$O$};

 \node at (-2.7,-1.7) {$K_1$}; 
\draw (0,0) --(-2.5,-0.57);
\fill[blue](-2.5,-0.57) circle (2pt);
\draw[ultra thick] (-2.1,-0.55) arc (0:340:0.5 and 0.8);

 \node at (-1.1,-1.7) {$K_2$}; 
\draw (0,0) -- (-1,-0.95);
\fill[blue] (-1,-0.95) circle (2pt);
\draw[ultra thick] (-0.85,-0.90) arc (3:340:0.2 and 0.4);

\draw (0,0) -- (0.6,-1.0) ;
\fill[blue] (0.6,-1) circle (2pt);
\draw[ultra thick] (0.5,-1.15) arc (-165:165:0.2 and 0.6);

 \node at (2.4,-1.8) {$K_n$}; 
\draw (0,0) -- (2.1,-0.73);
 \fill[blue] (2.1,-0.73) circle (2pt);
\draw[ultra thick] (1.75,-0.9) arc (-169:172:0.4 and 0.8);

\end{scope}

\end{tikzpicture} 
 \caption{The necklace $\mathcal{L}_n$.
 \label{fig:necklacebis}}
\end{figure}
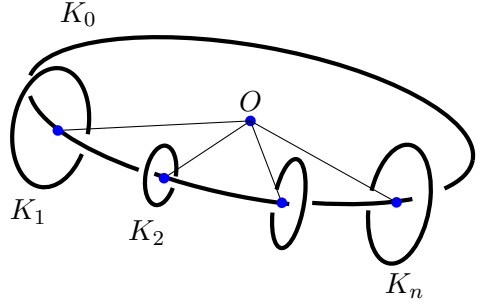

We will suppose $n\ge2$ and that the link is oriented. In particular each $K_i$ is a trivial knot such that:
\begin{equation*}
\left\{\begin{array}{ll}
\lk(K_0,K_i)=+1  &\quad i=1,\ldots,n \\
\lk(K_i,K_j)=0   &\quad i,j = 1,\ldots,n, \ \ i\neq j \\
\end{array}\right.
\end{equation*}

\subsection{The iteration $\tau^n$}


We denote by $\Conf \mathcal{L}_n$ the configuration space of \emph{unlabeled} necklaces
(i.e. two necklaces $K_0 \cup K_1 \cup  \ldots\cup K_n$ and  $K'_0 \cup K'_1 \cup  \ldots\cup K'_n$
are identified if they differ only by a permutation on the indices).  Since $\Conf \mathcal{L}_n$
is path-connected, we can  denote without ambiguity its fundamental group by  $\pi_1(\Conf \mathcal{L}_n)$.
Let $\tau \in \pi_1(\Conf \mathcal{L}_n)$ denote the circular permutation of 
circles $K_1 \to K_2$, $K_2 \to K_3$,\ldots such that any centre point $O_i$ is sent to  $O_{i+1}$ and any radius $r_i$ to $r_{i+1}$ (mod $n$).
It is important to understand the iteration $\tau^n$. 
We give two visions of $\tau^n$ in $\pi_1(\Conf \mathcal{L}_n)$.

 First, this move can of course be seen as $n$ iterations of $\tau$: so that it is a 
full rotation of all the $K_i$ ($i=1,\ldots,n$) along $K_0$, back to their initial position.
  
  On the other hand $\tau^n$ can be also seen as follows: suppose that $K_1,\ldots,K_n$ are 
  sufficiently closed circles. Fixing those $K_i$, $\tau^n$ corresponds to a full rotation of $K_0$
around all the $K_i$ (as $K_1,\ldots,K_n$ are closed we abusively suppose that their centres are aligned). 

Indeed each of these moves is a rotation of angle $2\pi$ around an axis, 
and we can continuously change the axis to go from the first vision to the second.

\begin{figure}
 \begin{tikzpicture}[scale=0.8]

\begin{scope}[rotate=10]
\draw[ultra thick] (3,0) arc (0:180:3 and 1);

\draw[gray,thick] (0,0)--+(0.1,0.05)--+(-0.1,-0.05);
\draw[gray,thick] (0,0)--+(-0.1,0.05)--+(0.1,-0.05);
\draw[gray,thick] (0,-0)--(0,3);

\draw[ultra thick] (3,0) arc (0:-30:3 and 1);
\draw[ultra thick] (3,0) arc (0:-30:3 and 1); 
\draw[ultra thick] (2.3,-0.65) arc (-40:-52:3 and 1) ;
\draw[ultra thick] (1.6,-0.85) arc (-59:-68:3 and 1) ;
\draw[ultra thick] (-3,0) arc (-180:-75:3 and 1) ;

\draw[ultra thick] (0.5,-1.2) arc (-165:165:0.3 and 0.6);
\draw[ultra thick] (1.2,-1.1) arc (-165:165:0.3 and 0.6);
\draw[ultra thick] (1.9,-0.9) arc (-165:165:0.3 and 0.6);

\draw[->,>=latex,ultra thick,blue] (2.5,0.2)  arc (-40:305:3 and 1);

\end{scope}

\end{tikzpicture}\qquad
 \begin{tikzpicture}[scale=0.8]

\begin{scope}[rotate=10]
\draw[ultra thick] (3,0) arc (0:180:3 and 1);
\draw[gray,thick] (-2.5,-1.7)--(3.65,-0.5);

\draw[gray,thick] (3.65,-0.5)--+(0.08,0.1)--+(-0.08,-0.1);
\draw[gray,thick] (3.65,-0.5)--+(0.05,-0.1)--+(-0.05,0.1);

\draw[ultra thick] (3,0) arc (0:-30:3 and 1);
\draw[ultra thick] (3,0) arc (0:-30:3 and 1); 
\draw[ultra thick] (2.3,-0.65) arc (-40:-52:3 and 1) ;
\draw[ultra thick] (1.6,-0.85) arc (-59:-68:3 and 1) ;
\draw[ultra thick] (-3,0) arc (-180:-75:3 and 1) ;

\draw[ultra thick] (0.5,-1.2) arc (-165:165:0.3 and 0.6);
\draw[ultra thick] (1.2,-1.1) arc (-165:165:0.3 and 0.6);
\draw[ultra thick] (1.9,-0.9) arc (-165:165:0.3 and 0.6);

\draw[ultra thick,blue] (3.5,0.32)  arc (125:-165:0.5 and 1);
\draw[<-,>=latex,ultra thick,blue] (3.45,0.25)  arc (125:170:0.5 and 1);

\end{scope}

\end{tikzpicture}
 \caption{Two visions of $\tau^n$. \label{fig:taun}}
\end{figure}
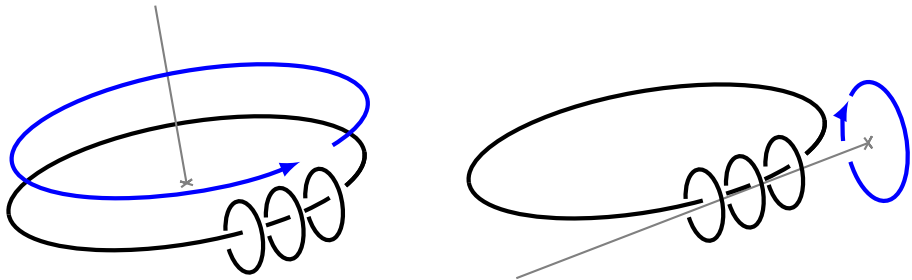

\subsection{The fundamental group of $\Conf \mathcal{L}_n$}
\label{ssec:pi1}

Before stating the main result  of this Section (Theorem \ref{th:circular}) we start with the particular case of    rigid configurations,  where
$K_0^*$ is the Euclidean circle in the plane $(Oxy)$ centered at $O$
and of radius $1$. A necklace having such a core circle $K_0^*$
is called a \defi{normalized necklace} and we denote it by $\mathcal{L}_n^*$.

\begin{lemma}
\label{lem:cbn}
$\Conf \mathcal{L}_n^*$ is homeomorphic to the configuration space $\mathcal{CB}_n$.  
\end{lemma}

\begin{proof}
We consider  $\mathcal{CB}_n$ as the configuration space of $n$
points lying on the annulus $A_0 = D_1 \setminus D_{\frac12}$, 
where $D_r$ denotes the closed disk in the plane $(Oxy)$ centered at $O$ of radius $r$.
To a normalized necklace $\mathcal{L}_n^*$ we associate $n$ points of $A_0$ as follows:
$$(K_1,\ldots,K_n) \longmapsto (K_1 \cap A_0, \ldots,K_n \cap A_0 ).$$
As each point of $A_0$ determines a unique normalized circle $K_i$,
this map is a bijection. This maps and its inverse are continuous.
\end{proof}  
  
As a consequence of previous Lemma we have that $\pi_1(\Conf \mathcal{L}_n^*)$ is isomorphic to $CB_n$. On the other hand in the following theorem we will prove that the inclusion $\Conf \mathcal{L}_n^* \to \Conf \mathcal{L}_n$ becomes a (proper) projection at the homotopy level.

\begin{theorem}
\label{th:circular}
For $n\ge 2$, the group  $\pi_1(\Conf \mathcal{L}_n)$ is isomorphic 
to the quotient $CB_n/\langle \zeta^{2n} \rangle$ of the circular braid group.  
\end{theorem}

\begin{proof}
To a necklace $\mathcal{L}_n$ having core circle $K_0$
we associate a unit vector $u_0$, orthogonal to the plane containing
$K_0$ (and oriented according to the orientation of $K_0$).

\begin{figure}
 \begin{tikzpicture}[scale=0.8]

\begin{scope}[rotate=-30]

\draw[ultra thick] (3,0) arc (0:87:3 and 1);
\draw[ultra thick] (-3,0) arc (0:87:-3 and 1);


\draw[->,>=latex,blue, thick] (0,-0)--(0,2) node[black, below right] {$u_0$};
\draw[->,>=latex,blue, thick] (0,-0)--(-55:0.8) node[midway, black,left] {$u_1$};
\fill[blue] (0,0) circle (3pt) node[black, right]{$O$};
\node[right] at (2,1) {$K_0$}; 
\node[right] at (-73:2.1) {$K_1$}; 

\draw[ultra thick] (3,0) arc (0:-30:3 and 1);
\draw[ultra thick] (3,0) arc (0:-30:3 and 1); 
\draw[ultra thick] (2.3,-0.65) arc (-40:-52:3 and 1) ;
\draw[ultra thick] (1.6,-0.85) arc (-59:-68:3 and 1) ;
\draw[ultra thick] (-3,0) arc (-180:-75:3 and 1) ;

\draw[ultra thick] (0.5,-1.2) arc (-165:165:0.3 and 0.6);
\draw[ultra thick] (1.2,-1.1) arc (-165:165:0.3 and 0.6);
\draw[ultra thick] (1.9,-0.9) arc (-165:165:0.3 and 0.6);


\end{scope}

\end{tikzpicture} 
 \caption{The normal vector $u_0$. \label{fig:uzero}}
\end{figure}
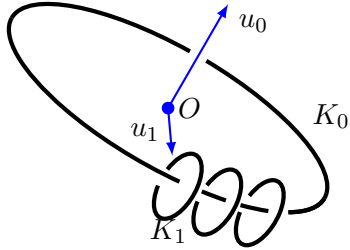

Let $G : \Conf(\mathcal{L}_n) \to S^2$ be the map defined
by $G(\mathcal{L}_n) = u_0$. This map $G$ is a locally trivial fibration
whose fiber is $\Conf \mathcal{L}_n^*$ (for the unit vector $u_0 = (0,0,1)$).
The long exact sequence in homotopy for the fibration 
$\Conf \mathcal{L}_n^* \hookrightarrow \Conf \mathcal{L}_n \twoheadrightarrow S^2$
provides:
$$
\begin{array}{c}
0 
\longrightarrow \pi_2(\Conf \mathcal{L}_n^*)
\overset{H_2}{\longrightarrow} \pi_2(\Conf \mathcal{L}_n)
\overset{G_2}{\longrightarrow} \pi_2(S^2) \\
\qquad\qquad \overset{d}{\longrightarrow} \pi_1(\Conf \mathcal{L}_n^*)
\overset{H_1}{\longrightarrow} \pi_1(\Conf \mathcal{L}_n)
\longrightarrow 0
\end{array}$$
It implies that $H_1$ is surjective (since in the exact sequence $\pi_1(S^2)$ is trivial).

Before computing the kernel of $H_1$, we give a motivation why $H_1$ is not injective.
There is a natural map $K$ from $\Conf \mathcal{L}_n$ to $SO_3$.
Let us see $SO_3$ as the space of  (right handed) orthonormal frames $(u_0,u_1,u_2)$.
To a necklace $\mathcal{L}_n$ we associate $u_0$ as above, while $u_1$ is the unit vector
from the origin $O$ to the center of $K_1$, then we set $u_2 = u_0 \wedge u_1$ (see figure \ref{fig:uzero}). 

Let us denote by $\pi : SO_3 \to S^2$,  the natural projection 
$\pi(u_0,u_1,u_2) = u_0$. 
We have a commutative diagram, that is to say $G = \pi \circ K$.
But as $\pi_2(SO_3) = 0$, it implies $G_2 = 0$. By the exact sequence, $d$ is injective, so that
$\pi_2(S^2) = \Zz \cong \Im d = \Ker H_1$.

There is another interesting point with $SO_3$. In fact if we now see $SO_3$ 
as the space of rotations, we denote 
by $\rho$ a full rotation around the vertical axis (supported by $u_0$).
Then $\rho \neq \id$, but $\rho^2 \simeq \id$ (because $\pi_1(SO_3) \cong \Zz/2\Zz$).
For us the move $\rho$ corresponds to the full rotation $\tau^n$. 
It gives the idea that in $\Conf \mathcal{L}_n^*$, $\tau^n$ generates a subgroup isomorphic to 
 $\Zz$, but $\tau^{2n}$ is homotopic to $\id$ in $\Conf \mathcal{L}_n$.

We will now compute the kernel of $H_1$ as $\Im d$, where $d$ is the boundary map 
coming from the exact sequence : $\Ker H_1 = \langle \tau^{2n} \rangle$.

We go back to the construction of this boundary map $d$ (see for instance,
\cite[Theorem 6.12]{Hu}) and we describe one of the liftings explicitly.
Let $f$ be the generator of $\pi_2(S^2) \cong \Zz$ defined by 
$f : S^2 \to S^2$, $f(x)=x$, but we prefer to see it as a map 
$f : I \times I \to S^2 \quad \text{such that}\quad f(\partial I^2) = N$
where $N=(0,0,1)$ is the North pole of $S^2$.

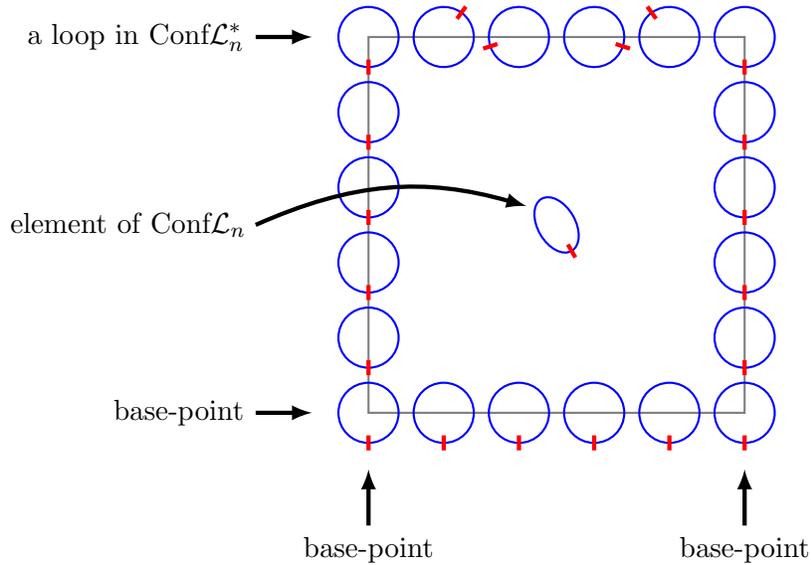
\begin{figure}
 \begin{tikzpicture}[scale=0.5]

\def\link{
  \draw[blue, thick] (0,0) circle(0.8); 
  \draw[ultra thick, red] (0,-0.6)--(0,-1);
}

\draw[thick, gray] (0,0) rectangle (10,10);

  \foreach \r in {0,2,...,10} {
      \begin{scope}[xshift=\r cm]
           \link
       \end{scope}
    }
  \foreach \r in {2,4,...,10} {
      \begin{scope}[yshift=\r cm]
           \link
       \end{scope}
    }
  \foreach \r in {2,4,...,10} {
      \begin{scope}[xshift=10cm,yshift=\r cm]
           \link
       \end{scope}
    }
  \foreach \r in {2,4,6,8} {
      \begin{scope}[xshift=\r cm, yshift = 10cm, rotate=72*\r]
           \link
       \end{scope}
    }

\begin{scope}[xshift=5 cm,yshift=5 cm, rotate=30]
  \draw[blue,thick] (0,0) ellipse(0.5 and 0.8); 
  \draw[ultra thick, red] (0,-0.6)--(0,-1);
\end{scope}

\draw[->,>=latex,ultra thick] (0,-3)--+(0,1.5) node[pos=0, below] {base-point};
\draw[->,>=latex,ultra thick] (10,-3)--+(0,1.5) node[pos=0, below] {base-point};
\draw[->,>=latex,ultra thick] (-3,0)--+(1.5,0) node[pos=0, left] {base-point};
\draw[->,>=latex,ultra thick] (-3,10)--+(1.5,0) node[pos=0, left] {a loop in $\text{Conf}\mathcal{L}_n^*$};
\draw[->,>=latex,ultra thick] (-3,5) to[bend left = 20] (4.2,5.5);
\node[left] at (-3,5) {element of $\text{Conf} \mathcal{L}_n$};

\end{tikzpicture}
 \caption{\label{fig:lift}The lifting $h$ from $I\times I$ to $\Conf \mathcal{L}_n$ ($K_0$ in blue, $K_1$ in red).}
\end{figure}

We lift $f$ to a map $\tilde f$ from $\partial I \times I \cup I \times \{0\}$
to the base-point $\mathcal{L}_n^*$ in $\Conf \mathcal{L}_n$ (and $\Conf \mathcal{L}_n^*$, for which 
$u_0$ is $\overrightarrow{ON}$).

By the homotopy lifting property, it extends to a map
$h : I\times I \to \Conf \mathcal{L}_n$,
and $d(f) \in \pi_1(\Conf \mathcal{L}_n^*)$ is the map 
induced by $h_{|I\times \{1\}}$ (figure \ref{fig:lift}).

One way to make this lift explicit  is to see $I^2 \setminus \partial I^2$ as $S^2 \setminus \{N\}$.
On $S^2$ there exists a continuous vector field, 
that is non-zero except at $N$ (see the dipole integral curves in figure \ref{fig:dipole}, see also \cite[section 5]{CaPa} 
for other links between vector fields and configuration spaces).
Now at each point $u_0 \in S^2 \setminus \{N\}$ is associated a non-zero vector $u_1$.
Let us define $h : S^2 \setminus \{N\} \to \Conf \mathcal{L}_n$
as follows: for $u_0 \in S^2\setminus \{N\}$, $h(u_0)$ is the unique
necklace $\mathcal{L}_n$, such $G(\mathcal{L}_n) = u_0$ and
the unit vector from $O$ to $K_1$ is $u_1$. We define $K_2,\ldots,K_n$ as parallel copies of $K_1$
(figure \ref{fig:neck}).
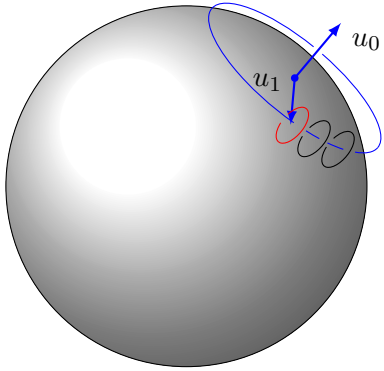
\begin{figure}
 \begin{tikzpicture}[scale=1.2]

\filldraw[ball color=white] (0,0) circle (2);

\begin{scope}[scale = 0.4, xshift=3cm,yshift=3cm, rotate=-40,]

\draw[blue] (3,0) arc (0:87:3 and 1);
\draw[blue] (-3,0) arc (0:87:-3 and 1);


\draw[->,>=latex,blue, thick] (0,-0)--(0,2) node[black, below right] {$u_0$};
\draw[->,>=latex,blue, thick] (0,-0)--(-55:1.3) node[midway, black,above left] {$u_1$};
\fill[blue] (0,0) circle (3pt); 

\draw[blue] (3,0) arc (0:-30:3 and 1); 
\draw[blue] (2.3,-0.65) arc (-40:-52:3 and 1) ;
\draw[blue] (1.6,-0.85) arc (-59:-68:3 and 1) ;
\draw[blue] (-3,0) arc (-180:-75:3 and 1) ;

\draw[red] (0.5,-1.2) arc (-165:165:0.3 and 0.6);
\draw (1.2,-1.1) arc (-165:165:0.3 and 0.6);
\draw (1.9,-0.9) arc (-165:165:0.3 and 0.6);


\end{scope}

\end{tikzpicture}
 \caption{\label{fig:neck}A necklace constructed from $u_0 \in S^2$ and $u_1 \in T_{u_0} S^2$.}
\end{figure}

Now $d(f)$ corresponds to the lift $h(\gamma)$ where $\gamma$ is a small loop in $S^2$ around $N$.
Due to the dipole figure at $N$, $h(\gamma)$ is exactly $\tau^{2n}$ (figures \ref{fig:vector} and \ref{fig:tau2n}).

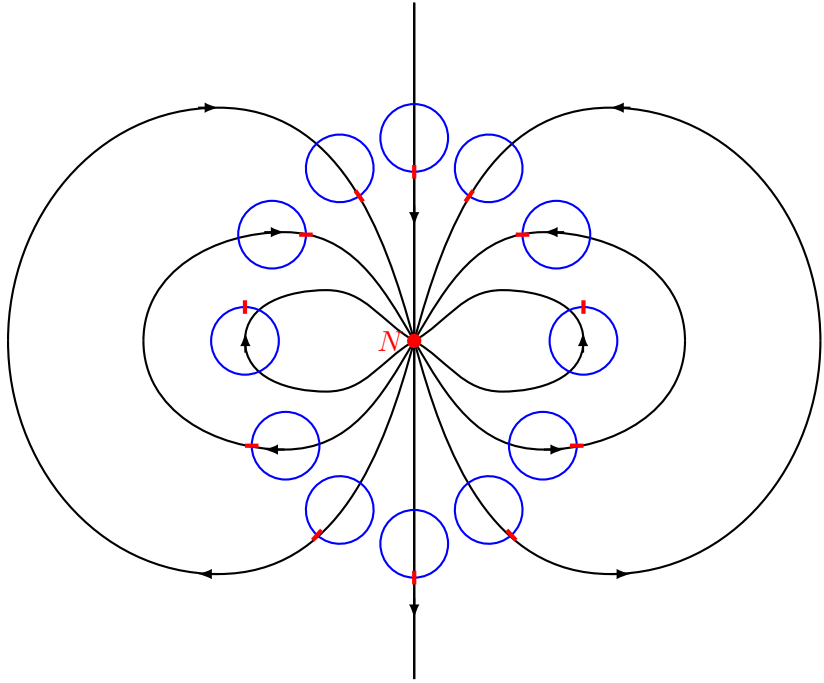
\begin{figure}
 \begin{tikzpicture}[scale=0.9]

\draw[thick] (0,0) to [out=-60,in=180] (-40:2.5) to [out=0,in=-90] (0:4);
\draw[thick] (0,0) to [out=-75,in=180] (-50:4.5) to [out=0,in=-90] (0:6);
\draw[thick] (0,0) to (-90:5);
\draw[thick] (0,0) to [out=-30,in=180] (-30:1.5) to [out=0,in=-90] (0:2.5);

\begin{scope}[cm={1,0,0,-1,(0,0)}]
  \draw[thick]  (0,0) to [out=-60,in=180] (-40:2.5) to [out=0,in=-90] (0:4);
  \draw[thick] (0,0) to [out=-75,in=180] (-50:4.5) to [out=0,in=-90] (0:6);
  \draw[thick] (0,0) to (-90:5);
  \draw[thick] (0,0) to [out=-30,in=180] (-30:1.5) to [out=0,in=-90] (0:2.5);
\end{scope}

\begin{scope}[cm={-1,0,0,-1,(0,0)}]
  \draw[thick]  (0,0) to [out=-60,in=180] (-40:2.5) to [out=0,in=-90] (0:4);
  \draw[thick] (0,0) to [out=-75,in=180] (-50:4.5) to [out=0,in=-90] (0:6);
  \draw[thick] (0,0) to (-90:5);
  \draw[thick] (0,0) to [out=-30,in=180] (-30:1.5) to [out=0,in=-90] (0:2.5);
\end{scope}

\begin{scope}[cm={-1,0,0,1,(0,0)}]
  \draw[thick]  (0,0) to [out=-60,in=180] (-40:2.5) to [out=0,in=-90] (0:4);
  \draw[thick] (0,0) to [out=-75,in=180] (-50:4.5) to [out=0,in=-90] (0:6);
  \draw[thick] (0,0) to (-90:5);
  \draw[thick] (0,0) to [out=-30,in=180] (-30:1.5) to [out=0,in=-90] (0:2.5);
\end{scope}

\draw[->,>=latex,thick] (-90:3.8)--+(0,-0.3);
\draw[->,>=latex,thick] (90:2)--+(0,-0.3);
\draw[->,>=latex,thick] (-4:2.5)--+(0,0.3);
\draw[->,>=latex,thick] (4:-2.5)--+(0,0.3);

\draw[->,>=latex,thick] (-40:2.5)--+(0.3,0);
\draw[<-,>=latex,thick] (40:2.5)--+(0.3,0);

\draw[->,>=latex,thick] (-50:4.5)--+(0.3,0);
\draw[<-,>=latex,thick] (50:4.5)--+(0.3,0);

\begin{scope}[cm={-1,0,0,1,(0,0)}]
\draw[->,>=latex,thick] (-40:2.5)--+(0.3,0);
\draw[<-,>=latex,thick] (40:2.5)--+(0.3,0);

\draw[->,>=latex,thick] (-50:4.5)--+(0.3,0);
\draw[<-,>=latex,thick] (50:4.5)--+(0.3,0);
\end{scope}

\fill[red] (0,0) circle (3pt) node[left]{$N$};

\def\link{
  \draw[blue, thick] (0,0) circle(0.5); 
  \draw[ultra thick, red] (0,-0.4)--(0,-0.6);
}
 \begin{scope}[yshift=-3 cm] \link  \end{scope}
 \begin{scope}[yshift=3 cm] \link  \end{scope}

 \begin{scope}[xshift = 1.1cm, yshift=-2.5 cm,rotate = 42] \link  \end{scope}
 \begin{scope}[xshift = 1.9cm, yshift=-1.55 cm,rotate = 90] \link  \end{scope}
 \begin{scope}[xshift = 2.1cm, yshift=1.57 cm,rotate = -90] \link  \end{scope}
 \begin{scope}[xshift = 1.1cm, yshift=2.55 cm,rotate = -35] \link  \end{scope}

 \begin{scope}[xshift = 2.5cm, yshift=0 cm,rotate = 180] \link  \end{scope}

 \begin{scope}[xshift = -1.1cm, yshift=-2.5 cm,rotate = -42] \link  \end{scope}
 \begin{scope}[xshift = -1.9cm, yshift=-1.55 cm,rotate = -90] \link  \end{scope}
 \begin{scope}[xshift = -2.1cm, yshift=1.57 cm,rotate = 90] \link  \end{scope}
 \begin{scope}[xshift = -1.1cm, yshift=2.55 cm,rotate = 35] \link  \end{scope}

 \begin{scope}[xshift = -2.5cm, yshift=0 cm,rotate = 180] \link  \end{scope}

\end{tikzpicture}
 \caption{\label{fig:vector}The vector fields on $S^2$ around $N$ and the family of necklaces.
 \label{fig:dipole}}
\end{figure}

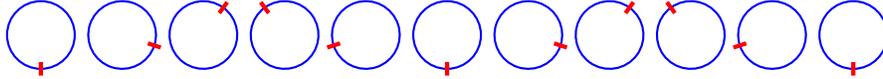
\begin{figure}
 \begin{tikzpicture}[scale=0.9]

\def\link{
  \draw[blue, thick] (0,0) circle(0.5); 
  \draw[ultra thick, red] (0,-0.4)--(0,-0.6);
}

  \foreach \r in {0,...,10} {
      \begin{scope}[xshift={1.2*\r cm},  rotate=72*\r]
           \link
       \end{scope}
    }

\end{tikzpicture}
 \caption{\label{fig:tau2n}The family of necklaces $\tau^{2n}$.}
\end{figure}

A first conclusion is that $\Ker H_1 = \langle d(f) \rangle = \langle \tau^{2n} \rangle$.
And finally : 

$$\pi_1(\Conf \mathcal{L}_n)
\cong 
\pi_1(\Conf \mathcal{L}_n^*) / \langle \tau^{2n} \rangle
\cong
CB_n/\langle \zeta^{2n} \rangle$$
\end{proof}

Let us end the section with a few remarks:

\begin{remark}
 Since $\zeta^n$ generates the center of $CB_n$ then the group generated  by  $\zeta^{2n}$ is normal: so we can effectively write $\langle \zeta^{2n} \rangle$
instead of  $\langle \langle \zeta^{2n} \rangle \rangle$; let us also recall that, denoting by $Mod_{n+2}(S^2)$  the mapping class group
of the $n+2$-punctured sphere,  $CB_n/\langle \zeta^{n} \rangle$ is isomorphic to the subgroup
of $Mod_{n+2}(S^2)$ fixing  two punctures (see \cite[Section 2]{CC} where $CB_n/\langle \zeta^{n} \rangle$ is denoted by $A(B_n)/Z$).
\end{remark}

\begin{remark} \label{rk:circular}
 In Theorem  \ref{th:circular}  we constructed a map $h_n: CB_n \to \pi_1(\Conf \mathcal{L}_n)$
where the generator $\zeta$ in $CB_n$ corresponds to  the move $\tau$ in $\pi_1(\Conf \mathcal{L}_n)$,
while   the generator $\sigma_i$ in $CB_n$ corresponds to the move $\sigma_i$ which permutes the $i$-th circle
with the $i+1$-th one (modulo $n$) by passing the $i$-th circle through the $i+1$-th. The kernel of $h_n$ 
is the group generated  by  $\zeta^{2n}$ and the group $\pi_1(\Conf \mathcal{L}_n)$ admits the following 
group presentation:

$$\pi_1(\Conf \mathcal{L}_n) = \left\langle \sigma_1,\ldots,\sigma_n, \tau \mid
\begin{array}{l}
\sigma_i\sigma_{i+1}\sigma_i = \sigma_{i+1}\sigma_i\sigma_{i+1} \; \text{ for } i = 1,\ldots, n, \\
\sigma_i\sigma_j=\sigma_j\sigma_i \quad \text{ for } |i-j| \neq 1, \\
\conjug{\tau} \sigma_i \tau =\sigma_{i+1} \quad \text{ for } i = 1,2\ldots, n \\
\tau^{2n}=1
\end{array}
\right\rangle
$$
\end{remark}


\section{Action on $\pi_1(\Rr^3\setminus \mathcal{L}_n$)}
\label{sec:action}

 As recalled in the introduction, we denote:
 \begin{itemize}
 \item  the configuration space 
  of $n$ unlinked Euclidean circles  by $\mathcal{R}_n$ and $\pi_1(\mathcal{R}_n)$ by $R_n$;
\item  the configuration space of $n$ unlinked Euclidean circles being all parallel to a fixed plane, say the $yz$-plane,
by $\mathcal{UR}_n$ and  its fundamental group by $WB_n$.
\end{itemize}

According to Proposition 2.2 of \cite{BH},  $WB_n$ can be seen as a subgroup of $R_n$ and it is generated by two families of elements, $\rho_i$ and $\sigma_i$ 
(see figure \ref{fig:moves} and \cite{BH}):  $\rho_i$ is the path permuting the $i$-th and the $i+1$-th circles by passing over,
while $\sigma_i$ permutes them by passing the $i$-th circle through the $i+1$-th.

A \defi{motion} of a compact submanifold $N$ in a  manifold $M$ is a path $f_t$
in $\text{Homeo}_c(M)$ such that $f_0=\textup{id}$ and $f_1(N)=N$, where $\text{Homeo}_c(M)$
denotes  the group of homeomorphisms of $M$ with compact support. A
motion is called \defi{stationary} if $f_t(N)=N$ for all $t\in[0,1]$. 

The \defi{motion
group} $\mathcal{M}(M,N)$ of $N$ in $M$ is the group of equivalence classes of
motions of $N$ in $M$  where two motions $f_t,g_t$ are equivalent if
$g_t^{-1}f_t$ is homotopic relative to endpoints to a stationary motion.
The notion of motion groups was proposed by R.~Fox, and studied
by P.~Dahm, one of his students. The first published article on the topic is \cite{G1}.

Notice that motion groups generalize fundamental groups of configuration spaces,
and  in our cases each motion is equivalent to a motion that fixes
a point $*\in M \setminus N$. When $M$ is non-compact, it is possible to define a homomorphism
(the \emph{Dahm morphism}):
\begin{equation*}
D_N: \mathcal{M}(M,N)\to \Aut(\pi_1(M \setminus N,*))
\end{equation*}
 sending an element represented by the motion $f_t$, into the automorphism
induced on $\pi_1(M\setminus N,*)$ by $f_1$.

When $M=\Rr^3$ and $N$ is a set $\mathcal{L}'_n$ of $n$ unlinked Euclidean circles 
we get a map 
\begin{equation*}
D_{\mathcal{L}'_n}:R_n \to \Aut(\pi_1(\Rr^3 \setminus \mathcal{L}'_n,*))
\end{equation*}
This map is injective (see \cite{G1}) and sends generators of $R_n$ (and therefore of $WB_n$) to
 the following  automorphisms of the free group $F_n=\langle x_1, \ldots, x_n \rangle$: 
$$\sigma_i : 
\left\{\begin{array}{l}
x_i \mapsto x_{i} x_{i+1} \conjug{x_{i}} \\
x_{i+1} \mapsto x_i \\
x_{j} \mapsto x_j \quad j \neq i,i+1\\     
\end{array}\right.
\; 
\rho_i : 
\left\{\begin{array}{l}
x_i \mapsto   x_{i+1}   \\
x_{i+1} \mapsto x_i \\
x_{j} \mapsto x_j \quad  j \neq i,i+1\\       
\end{array}\right.
$$
$$
\tau_j : 
\left\{\begin{array}{l}
x_j \mapsto   \conjug{x_j}   \\
x_{k} \mapsto x_k \quad   k \neq j\\       
\end{array}\right.$$ 
where $i=1,\ldots, n-1$ and $j=1,\ldots, n$.
 
Now let  $\mathcal{L}_n$ be a necklace: by forgetting 
the core circle $K_0$ we obtain a map from $\Conf \mathcal{L}_n$ to $\mathcal{R}_n$.
To  $\Gamma = \mathcal{L}_n(t)$ in $\pi_1(\Conf \mathcal{L}_n)$,
we can associate an automorphism 
$D_{\mathcal{L}_n}(\Gamma) : \pi_1(\Rr^3\setminus \mathcal{L}_n) 
\to \pi_1(\Rr^3\setminus \mathcal{L}_n)$.   It is easy to  compute the $\pi_1$ of the complement of $\mathcal{L}_n$ by giving its Wirtinger presentation:
$$\pi_1(\Rr^3 \setminus \mathcal{L}_n) = \left\langle x_1,\ldots,x_n,y \mid x_iy=yx_i, i=1\ldots,n \right\rangle.$$
A direct justification, suggested by the referee, is the following:
consider $\mathcal{L}_n$ as a link in $S^3$ (instead of $\Rr^3$) then
$S^3\setminus K_0$ is homeomorphic to the torus $S^1\times \Rr^2$ and 
$S^3\setminus\mathcal{L}_n$ is homeomorphic to $S^1 \times (\Rr^2 \setminus \{ n \text{ points}\})$.

So that we obtain
$$\pi_1(\Rr^3 \setminus \mathcal{L}_n) \cong F_n \times \Zz.$$
The action of the generators of $\pi_1(\Conf \mathcal{L}_n)$ are (with indices  modulo $n$):
$$\sigma_i : 
\left\{\begin{array}{l}
x_i \mapsto x_{i} x_{i+1} \conjug{x_{i}} \\
x_{i+1} \mapsto x_i \\
x_{j} \mapsto x_j \quad j \neq i,i+1\\
y \mapsto y         
\end{array}\right.
\qquad 
\tau : 
\left\{\begin{array}{l}
x_j \mapsto x_{j+1} \\
y \mapsto y         
\end{array}\right.$$

Notice that the action of $\pi_1(\Conf \mathcal{L}_n)$ on $F_n \times \Zz$  is actually  the product of
an action on $F_n$  and the identity action on $\Zz$. The following Lemma characterizes therefore the kernel 
of the action of $\pi_1(\Conf \mathcal{L}_n)$ on $F_n$.

\begin{lemma}
\label{lem:dahm}

\begin{enumerate}
\item Let  $D_{\mathcal{L}_n} : \pi_1(\Conf \mathcal{L}_n) \to \Aut  \pi_1(\Rr^3\setminus \mathcal{L}_n)$
be  the  Dahm morphism;
then $\Ker D_{\mathcal{L}_n} = \langle \tau^n \rangle = \Zz / 2\Zz$.
\item Let $\Phi$ be the natural map $\Phi : \pi_1(\Conf \mathcal{L}_n) \to  R_n$ 
induced by forgetting the core circle $K_0$. Then $\Ker \Phi = \langle \tau^n \rangle = \Zz/ 2\Zz$.
\end{enumerate}
\end{lemma}

\begin{proof}
We first notice some facts:

(a) The following diagram is commutative :
$$ \xymatrix{
{}\pi_1(\Conf \mathcal{L}_n) \ar[r]^-{D_{\mathcal{L}_n}}\ar[d]_-{\Phi}  & \Aut  \pi_1(\Rr^3\setminus \mathcal{L}_n) \ar[d]^-{\Psi} \\
  R_n \ar[r]_-{D_{\mathcal{L}'_n}}        &  \Aut  \pi_1(\Rr^3\setminus \mathcal{L}'_n)
}
$$
where $\Phi$ and $\Psi$ are natural maps induced by the inclusion $\mathcal{L}'_n \subset \mathcal{L}_n$.
Remark that  $\Psi$ is then the  map which forgets the generator $y$ in $ \pi_1(\Rr^3\setminus \mathcal{L}_n)$ and therefore 
$$\psi(D_{\mathcal{L}_n})(\sigma_i) : 
\left\{\begin{array}{l}
x_i \mapsto x_{i} x_{i+1} \conjug{x_{i}} \\
x_{i+1} \mapsto x_i \\
x_{j} \mapsto x_j \quad j \neq i,i+1       
\end{array}\right.
\qquad 
\psi(D_{\mathcal{L}_n})(\tau) : 
\left\{\begin{array}{l}
x_j \mapsto x_{j+1}        
\end{array}\right.$$
where indices are modulo $n$.
Comparing with Theorem \ref{thmCB} and \ref{th:circular} we deduce that the $ \Psi \circ D_{\mathcal{L}_n} \circ h_n = \rho_{CB}$
and  $\Ker \Psi \circ D_{\mathcal{L}_n} \circ h_n = \langle \zeta^n \rangle$ 
(recall that  $h_n : CB_n \to \pi_1(\Conf \mathcal{L}_n)$ was defined  in \ref{rk:circular}).

(b) As we already recall, it is known that for the trivial link $\mathcal{L}'_n$ 
the Dahm morphism $D_{\mathcal{L}'_n}$ is injective (\cite{G1}).

(c) $\Psi$ is  injective when restricted to the image of $D_{\mathcal{L}_n}$.
If $\Psi(f)=\id$, then $f(x_i)=x_i$ for all $i=1,\ldots,n$.
If $f \in \Im D_{\mathcal{L}_n}$, then due to the action of the generators ($\sigma_i$ and $\tau$), it implies $f(y)=y$.
Finally if $\Psi(f)=\id$ and $f \in \Im D_{\mathcal{L}_n}$, then $f=\id$. Then $\Ker D_{\mathcal{L}_n} \circ h_n = \langle \zeta^n \rangle$.
Clearly, $ \tau^n \in \Ker D_{\mathcal{L}_n} $; on the other hand 
since $h_n$ is surjective, if $x \in \Ker D_{\mathcal{L}_n} $, 
then $x\in \Im (\langle \zeta^n \rangle)$ and therefore $x \in \langle \tau^n \rangle$
(see Remark \ref{rk:circular}).

For the second statement it is therefore enough to prove that the kernel of 
$D_{\mathcal{L}_n}$ coincides with  the kernel of $\Phi$.
$$
\begin{array}{rcll}
\gamma \in \Ker \Phi 
& \iff & \Phi(\gamma)= \id \\
& \iff & D_{\mathcal{L}'_n} \circ \Phi (\gamma) = \id &\qquad \text{because $D_{\mathcal{L}'_n}$ is injective} \\
& \iff & \Psi \circ D_{\mathcal{L}_n}(\gamma) = \id &\qquad \text{because the diagram commutes} \\
& \iff & D_{\mathcal{L}_n}(\gamma) = \id &\qquad \text{because $\Psi_{| D_{\mathcal{L}_n}}$  is injective}\\
& \iff & \gamma \in \Ker  D_{\mathcal{L}_n}(\gamma)&\\
\end{array}
$$

\end{proof}

\section{Characterization of automorphisms}
\label{sec:auto}

We will use the following notation: for a word $\conjug{w}= w^{-1}$, 
for an automorphism $\conjug{\alpha}=\alpha^{-1}$.
A famous result of Artin, characterizes automorphisms induced by braids.
\begin{theorem}[Artin]
The automorphisms induced by the action of $B_n$ on $F_n$ are exactly 
the automorphisms $\phi$ of $\Aut F_n$ that verify the two conditions below:
\begin{equation}
\label{eq:artinconj1}
\phi(x_i)=w_i x_{\pi(i)} \conjug{w_i}  
\end{equation}
for some $w_1,\ldots,w_n \in F_n$ and some permutation $\pi \in \mathcal{S}_n$, and:
\begin{equation}
\label{eq:artinconj2}
\phi(x_1x_2\cdots x_n) = x_1x_2\cdots x_n
\end{equation}
\end{theorem}

Interestingly, if we do not require condition (\ref{eq:artinconj2}), we recover 
exactly automorphisms of $\Aut F_n$ induced by welded braids. Recall 
that the welded braid group is generated by two types of moves $\sigma_i, \rho_i$, which induced 
two kinds of automorphisms also denoted $\sigma_i, \rho_i$ and described in section \ref{sec:action}.

\begin{theorem}[Theorem 4.1 of \cite{FRR}]
The automorphisms of $\Aut F_n$ induced  by the action of $WB_n$ on $F_n$ are exactly  those verifying (\ref{eq:artinconj1}).
\end{theorem}

 As a straightforward consequence we have
 that the natural map $B_n \to WB_n$
sending $\sigma_i$ into $\sigma_i$ is injective. We will show in  section \ref{ssec:linearnecklace} a geometric interpretation of
such an embedding.

Now we relax condition (\ref{eq:artinconj2}) and characterize automorphisms induced by our configurations.
\begin{theorem}
\label{th:circularartin}
The automorphisms induced by the action  of
$\pi_1(\Conf \mathcal{L}_n)$ on $F_n$ are exactly 
the automorphisms $\phi$ of $\Aut F_n$ that verify the two conditions below:
\begin{equation}
\label{eq:conj1}
\phi(x_i)=w_i x_{\pi(i)} \conjug{w_i}  
\end{equation}
for some $w_1,\ldots,w_n \in F_n$ and some permutation $\pi \in \mathcal{S}_n$, and:
\begin{equation}
\label{eq:conj2}
\phi(x_1x_2\cdots x_n) = w  x_1x_2\cdots x_n \conjug{w} 
\end{equation}
for some $w \in F_n$.
\end{theorem}

\begin{proof}
Let  us start with some notations.

We will denote by $\mathcal{A}_n$ the set of automorphisms of $\Aut F_n$ induced
by $\pi_1(\Conf \mathcal{L}_n)$ and $\mathcal{B}_n$ the set of automorphisms of $\Aut F_n$
that verify conditions (\ref{eq:conj1}) and (\ref{eq:conj2}); we will prove that $\mathcal{A}_n=\mathcal{B}_n$.

 We set also  $\Delta = x_1x_2\cdots x_n$, and, for any $w\in F_n$, we set the automorphism
$g_w \in \Aut F_n$  to be  $g_w(x_i) = w x_i \conjug{w}$ and therefore for any  $w'\in F_n$, we have 
that $g_w(w')= w w' \conjug{w}$.

First of all, the action of $\pi_1(\Conf \mathcal{L}_n)$ on $\Aut F_n$
is generated by the automorphisms $\sigma_i$ ($i=1,\ldots,n$) and $\tau$
that verify equations (\ref{eq:conj1}) and (\ref{eq:conj2}). 
In fact for $i=1,\ldots,n-1$, $\sigma_i(\Delta)=\Delta$ ; 
$\sigma_n(\Delta) =x_n \conjug{x_1} \Delta x_1 \conjug{x_n}$, $\tau(\Delta)=\conjug{x_1} \Delta x_1$.
It proves that $\mathcal{A}_n \subset \mathcal{B}_n$.

The remaining part is to prove that $\mathcal{B}_n \subset \mathcal{A}_n$: given an automorphism $f \in \Aut F_n$ that verifies
conditions (\ref{eq:conj1}) and (\ref{eq:conj2}), we express it as the automorphism induced by
some element of $\pi_1(\Conf \mathcal{L}_n)$.

A simple verification proves the following equation: the automorphism $g_{x_1}$ 
defined by $x_i \mapsto x_1 x_i \conjug{x_1}$ is generated by
elements of $\mathcal{A}_n$.
$$g_{x_1} = \sigma_1 \circ \sigma_2 \circ \cdots \circ \sigma_{n-1} \circ \conjug{\tau}$$

Previous remark allows to us to prove that 
$g_{x_k} \in \mathcal{A}_n$ since we have that:
$$g_{x_k} = \underbrace{\tau\circ \tau \circ \cdots \tau}_{k-1 \text{ occurrences}} \circ g_{x_1} \circ 
\underbrace{\conjug{\tau}\circ\conjug{\tau} \circ \cdots \conjug{\tau}}_{k-1 \text{ occurrences}}$$
We also generate $g_{x_k^{-1}}$ as the inverse of $g_{x_k}$.

Now we have to prove that a general element belongs to $\mathcal{A}_n$. More precisely, let $w \in F_n$; we will
generate the automorphism $g_w$ by induction on the length of $w$.
Suppose that $w = x_k w'$ with $w' \in F_n$ of length strictly less than the length of $w$.
Suppose that $g_{w'} \in  \mathcal{A}_n$. Then
$$g_w = g_{x_k} \circ g_{w'} \in \mathcal{A}_n.$$

The next step of the proof is to simplify the action on $\Delta$.
Let $f\in \Aut F_n$. Suppose that $f$ verifies conditions (\ref{eq:conj1}) and (\ref{eq:conj2}).
In particular, let $w \in F_n$ such that $f(\Delta)= w \Delta \conjug{w}$.
Then $g_{\conjug{w}} \circ f$ still satisfies conditions of type (\ref{eq:conj1})
and the condition (\ref{eq:artinconj2}):  $g_{\conjug{w}} \circ f(\Delta)=\Delta$.

Since  $g_{\conjug{w}} \circ f$ verifies condition (\ref{eq:artinconj1}) (which is exactly condition (\ref{eq:conj1})) and condition (\ref{eq:artinconj2}),
by Artin's theorem $g_{\conjug{w}} \circ f \in \mathcal{A}_n$ and therefore also  $f \in \mathcal{A}_n$.
It ends the proof of $\mathcal{B}_n \subset \mathcal{A}_n$, so that we get 
$\mathcal{A}_n=\mathcal{B}_n$.

\end{proof}

\section{Zero angular sum}
\label{sec:zero}

We say that $\Gamma \in \pi_1(\Conf \mathcal{L}_n)$
has \defi{zero angular sum} if $\Gamma \in \langle \sigma_1,\ldots,\sigma_n\rangle$.
This definition is motivated by the fact that a move $\sigma_i$
shifts the component $K_i$ by an angle of --say-- $+\frac{2\pi}{n}$
while $K_{i+1}$ is shifted by $-\frac{2\pi}{n}$, the sum of these angles being zero.
On the other hand $\tau$, moves each $K_i$ by an angle of --say-- $+\frac{2\pi}{n}$,
with a sum of $2\pi$.
The aim of this section is to characterize the zero angular sum condition 
at the level of automorphisms. We will define a kind of total winding number 
$\epsilon(\Gamma)$ about the axis of rotation of $K_0$.

Let $\epsilon : \pi_1(\Conf \mathcal{L}_n) \to \Zz$
defined as follows: to $\Gamma \in \pi_1(\Conf \mathcal{L}_n)$,
we associate the automorphism $\overline{\phi} = D_{\mathcal{L}_n}(\Gamma)$
by the Dahm morphism. Since the action on the generator $y$ of $\pi_1(\Rr^3\setminus \mathcal{L}_n)$ is the identity
(see remark before Lemma \ref{lem:dahm}),  $\overline{\phi}$
induces an automorphism $\phi$ of $\Aut F_n$ (obtained by setting $y=1$).
By theorem \ref{th:circularartin}, $\phi(x_1x_2\cdots x_n) = w  x_1x_2\cdots x_n \conjug{w}$,
for some $w \in F_n$. We define $\epsilon(\Gamma) = \ell(w) \in \Zz$
to be the algebraic length of the word $w$. 
We have the following characterization of zero angular sum:
\begin{proposition}
\label{prop:angmom}
$\Gamma \in \langle \sigma_1,\ldots,\sigma_n\rangle$
if and only if $\epsilon(\Gamma) = 0$.
\end{proposition}

\begin{proof}
 If we denote $\Delta = x_1\cdots x_n$ then for
 $i=1,\ldots,n-1$, $\sigma_i(\Delta)=\Delta$; on the other hand 
$\sigma_n(\Delta) =x_n \conjug{x_1} \Delta x_1 \conjug{x_n}$ and $\tau(\Delta)=\conjug{x_1} \Delta x_1$.
This implies  that $\epsilon(\sigma_i)=0$ and $\epsilon(\tau)=1$ (recall that $\epsilon$ is a homomorphism).

 Any $\Gamma \in  \pi_1(\Conf \mathcal{L}_n)$ can be written as
$\Gamma = \tau^k \sigma_{i_1}\cdots \sigma_{i_\ell}$ (by using the relations 
$\sigma_i \tau = \tau\sigma_{i+1}$).
Hence,  $\epsilon(\Gamma)=0$ implies  $k=0$, in which case
 $\Gamma \in  \langle \sigma_1,\ldots,\sigma_n\rangle$.

\end{proof}

We recall that the \defi{affine braid group}  $\tilde{A}_{n-1}$, is the group obtained by the group presentation of   $B_{n+1}$ 
 by replacing the relation $\sigma_{n} \sigma _1 =\sigma _1\sigma_{n}$ with the relation $\sigma_{n} \sigma _1\sigma_{n} =\sigma _1\sigma_{n} \sigma _1$.
By comparison of group presentations one deduces that   $CB_n = \tilde{A}_{n-1} \rtimes \langle\zeta\rangle$ (see also \cite{CC,KP}).  It then follows from Theorem \ref{th:circular}
and Remark  \ref{rk:circular} that  $\pi_1(\Conf \mathcal{L}_n)= \tilde{A}_{n-1} \rtimes \langle\tau\rangle$.  Since $\tau$ has finite order, $\pi_1(\Conf \mathcal{L}_n)$
inherits some properties of  $\tilde{A}_{n-1}$, in particular:

\begin{corollary}
The group $\pi_1(\Conf \mathcal{L}_n)$ is linear and, provided with the group presentation given in Remark \ref{rk:circular}, has solvable word problem.
\end{corollary}

On the other hand, it follows from Theorem \ref{th:circular} and Proposition \ref{prop:angmom} that:

 \begin{proposition}
 The affine braid group on $n$ strands, $\tilde{A}_{n-1}$ is isomorphic to  the subgroup of  $\pi_1(\Conf \mathcal{L}_n)$
 consisting of elements of zero angular sum.
 \end{proposition}

Consider now the representation $\rho_{\Aff} :  \tilde{A}_{n-1} \longrightarrow \Aut F_n$ 
induced by the action of $\pi_1(\Conf \mathcal{L}_n)$ on $F_n$ (obtained by setting $y=1$).

 $$ i\not=n:
\rho_{\Aff} (\sigma_{i}) : \left\{
\begin{array}{ll}
x_{i} \longmapsto x_{i} \, x_{i+1} \, \conjug{x_i}, &  \\ 
x_{i+1} \longmapsto x_{i}, & \\ 
x_{j} \longmapsto x_{j}, &  j\neq i,i+1.
\end{array} \right.
$$
$$
\rho_{\Aff} (\sigma_{n}) : 
\left\{
\begin{array}{ll}
x_{1} \longmapsto  x_{n}  &  \\
x_{n} \longmapsto
x_{n} x_{1}  \conjug{x_{n}} , & \\ 
x_{j} \longmapsto x_{j}, &  j\neq 1, n.
\end{array} \right.
$$

\begin{theorem}\label{thm:affine}

\begin{enumerate}[i)]
\item The representation  $\rho_{\Aff} :  \tilde{A}_{n-1} \longrightarrow \Aut F_n$  is faithful. 
\item An element $\phi \in Aut F_n$ belongs to $\rho_{\Aff} (\tilde{A}_{n-1})$ if and only if it verifies the conditions below:
\begin{equation}
\label{eq:conj1bis}
\phi(x_i)=w_i x_{\pi(i)} \conjug{w_i}  
\end{equation}
for some $w_1,\ldots,w_n \in F_n$ and some permutation $\pi \in \mathcal{S}_n$, and:
\begin{equation}
\label{eq:conj2bis}
\phi(x_1x_2\cdots x_n) = w  x_1x_2\cdots x_n \conjug{w} 
\end{equation}
for some $w \in F_n$ with algebraic length $\ell(w)=0$.
\end{enumerate}
 \end{theorem}

\begin{proof}
The kernel of $\rho_{\Aff} $ is a subgroup of the kernel of $\Psi \circ D_{\mathcal{L}_n} \circ h_n$, which is generated by $\tau^n$ (Lemma \ref{lem:dahm}).
Since $\tau^n$ generates the center of $\pi_1(\Conf \mathcal{L}_n)$, the kernel of  $\rho_{\Aff} $ is a subgroup of the center of  $\tilde{A}_{n-1}$.
Since the center of $\tilde{A}_{n-1} $ is trivial (see for instance \cite{JA}) we can conclude that $\rho_{\Aff}$ is faithful. The characterization given in the second statement follows 
 combining Theorem \ref{th:circularartin} and Proposition \ref{prop:angmom}.
\end{proof}

\section{The linear necklace}
\label{ssec:linearnecklace}

Let $\mathcal{C}^*_n =  K_1 \cup\ldots\cup K_n$ be the link where
each $K_i$ is a Euclidean circle  parallel to the $yz$-plane and centered at the $x$-axis.
We call such a  link a \defi{linear necklace}, thinking of the $x$-axis as $K_0$, a circle
passing through a point at infinity.
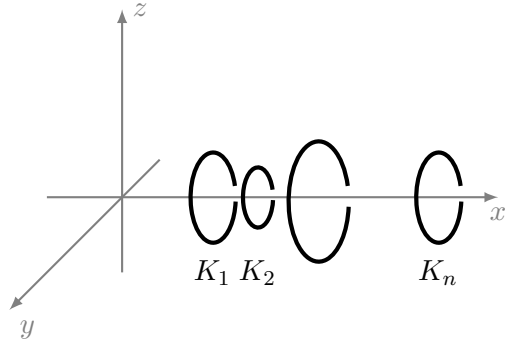
\begin{figure}
 \begin{tikzpicture}[scale=1]

\begin{scope}

\draw[->,>=latex,gray,thick] (-1,0)--(5,0) node[below]{$x$};
\draw[->,>=latex,gray,thick] (0,-1)--(0,2.5) node[right]{$z$};
\draw[->,>=latex,gray,thick] (0.5,0.5)--(-1.5,-1.5) node[below right]{$y$};

\draw[ultra thick] (1.5,0.15) arc (15:355:0.3 and 0.6);
\draw[ultra thick] (2,0.1) arc (15:355:0.2 and 0.4);
\draw[ultra thick] (3,0.15) arc (15:355:0.4 and 0.8);
\draw[ultra thick] (4.5,0.15) arc (15:355:0.3 and 0.6);
 \node at (1.2,-1) {$K_1$}; 
 \node at (1.8,-1) {$K_2$}; 
 \node at (4.2,-1) {$K_n$}; 
\end{scope}

\end{tikzpicture}
 \caption{A linear necklace. \label{fig:linearnecklace}}
\end{figure}

We recall that $\mathcal{UR}_n$ is the configuration space  of $n$ 
disjoint Euclidean circles lying on planes parallel to the $yz$-plane.
We have that:
\begin{theorem}
The inclusion of  $\Conf \mathcal{C}^*_n$ into $\mathcal{UR}_n$ induces an injection at the level 
of fundamental groups.
\end{theorem}
\begin{proof}
Let us remark that the moves $\sigma_1, \ldots, \sigma_{n-1}$ depicted in figure 2 belong to $\pi_1(\Conf \mathcal{C}^*_n) $. Actually, they
generate $\pi_1(\Conf \mathcal{C}^*_n) $; in fact, the position of any circle $K_i$ of $\mathcal{C}^*_n$
is determined by the intersection with the half-plane $y=0$ and $z>0$.
It follows that the configuration space of linear necklaces, $\Conf \mathcal{C}^*_n$, 
can be identified with the configuration space of $n$ distinct points in the (half-)plane, 
so that $\pi_1(\Conf \mathcal{C}^*_n) =B_n$,  where  generators are  exactly moves $\sigma_1, \ldots, \sigma_{n-1}$.
\end{proof}

The previous result provides then a geometric interpretation for the algebraic embedding of 
$B_n$ into $WB_n$ as subgroups of $\Aut(F_n)$.

\subsection*{Pure subgroups.}

Let us denote by $\Conf_{\text{Ord}}\mathcal{C}_n$ the configuration space of $n$ 
  ordered disjoint Euclidean circles lying on planes parallel to the $yz$-plane:
 $\pi_1(\Conf_{\text{Ord}} \mathcal{C}_n)$ is called the \defi{pure welded  braid group} on $n$ strands and will be denoted by $WP_n$,
 while $\pi_1(\Conf_{\text{Ord}} \mathcal{C}^*_n)$ is isomorphic to the pure braid group $P_n$.  Previous results imply that $P_n$ embeds geometrically in 
 $WP_n$.
 
More precisely, the group  $\pi_1(\Conf_{\text{Ord}} \mathcal{C}^*_n)$  is generated 
 by the family of paths $\lambda_{i,j}$ for $1\le i < j \le n$:  $\lambda_{i,j}$ moves 
 only the $i$-th circle that  passes inside  the following ones until the $j-1$-th one, 
  inside-outside the  $j$-th one and that  finally comes back passing inside the other circles.
 
Notice also that  in \cite{BH}, Brendle and Hatcher  introduced the configuration spaces of circles lying on parallel planes of different size, that 
 we can denote by $\mathcal{UR}^<_n$.  We can take as base-point for $\pi_1(\mathcal{UR}^<_n)$  a configuration
  of parallel circles with center on the $z$-axis and such that for any $i=1, \ldots, n-1$
 the $i$-th circle has radius greater than the radius of the $i+1$-th one. Let us remark that all other choices of base point  give conjugated subgroups  in  $WP_n$
 corresponding to different  permutations of circles.   As shown in \cite{BH},  $\pi_1(\mathcal{UR}^<_n)$ is generated by $\delta_{i,j}$ for $1\le i < j \le n$:  $\delta_{i,j}$ moves only the $i$-th circle that 
  passes outside  the $j-1$-th one (without passing inside-outside the other circles) and moves back (without passing inside-outside the other circles). 
  
 Let us recall that $\pi_1(\mathcal{UR}^<_n)$ is  called \defi{upper McCool group}
 in \cite{CPVW} and denoted by $P\Sigma_n^+$:   it is interesting to remark  that $P\Sigma_n^+$ and $P_n$ have isomorphic Lie algebras associated to the lower central series
  and the groups themselves are isomorphic for $n=2,3$ (\cite{CPVW}). A.~Suciu communicated to  us that using ideas from \cite{CoS}  it is possible to show that
  the ranks  of \emph{Chen groups} are different for $n>3$ and therefore $P\Sigma_n^+$ and $P_n$ are not isomorphic for 
 $n>3$.

\section{Proof of Theorem \ref{thmCB}}\label{2.1proof}
We have to proof 
that the kernel of $\rho_{CB}:  CB_n \to \Aut F_n$ is the cyclic group generated by $\zeta^n$.

  Let us recall that when we restrict the map $D_{\mathcal{L}'_n}:R_n \to \Aut(\pi_1(\Rr^3 \setminus \mathcal{L}'_n,*))$ to the braid group $B_n$
we get the usual Artin representation $\rho_A : B_n \to \Aut F_n$:
$$\rho_A(\sigma_i) : 
\left\{\begin{array}{l}
x_i \mapsto x_{i} x_{i+1} \conjug{x_{i}} \\
x_{i+1} \mapsto x_i \\
x_{j} \mapsto x_j \quad j \neq i,i+1\\    
\end{array}\right.
$$
for  the usual generators $\sigma_1, \ldots, \sigma_{n-1}$  of $B_n$.

Therefore we can define the group   $B_n \ltimes_{\rho_A} F_n$:
as generators  we will  denoted by $\alpha_1, \ldots, \alpha_{n-1}$ the generators of the factor
$B_n$ and  by $\eta_1,\ldots \eta_n$ a set  of generators for $F_n$.
According to such a set of generators a possible complete set of relations is the following:
$$
\begin{array}{l}
\alpha_i\alpha_{i+1}\alpha_i = \alpha_{i+1}\alpha_i\alpha_{i+1} \quad \text{ for } i = 1,2\ldots, n-1, \\
\alpha_i\alpha_j=\alpha_j\alpha_i \quad \text{ for } |i-j| \neq 1, \\
\alpha_i^{-1} \eta_i \alpha_i =\eta_{i}\eta_{i+1}\eta_{i}^{-1} \quad \text{ for } i = 1,2\ldots, n-1 \\
\alpha_i^{-1} \eta_{i+1} \alpha_i =\eta_{i} \quad \text{ for } i = 1,2\ldots, n-1 \\
\alpha_i^{-1} \eta_k \alpha_i =\eta_{k}  \quad \text{ for } i = 1,2\ldots, n-1 \quad \text{ and }  k \not= i, i+1
\end{array}
$$

 The group $CB_n$ is isomorphic to $B_n \ltimes_{\rho_A} F_n$ (\cite{CrP});
we leave to the reader the verification that an isomorphism is given by the map
$\Theta_n : CB_n \to B_n \ltimes_{\rho_A} F_n$  defined as follows:
$\Theta_n(\zeta)=\sigma_{n-1} \cdots \sigma_1 \eta_1$, $\Theta_n(\sigma_j)=\alpha_j$ for $j=1, \ldots, n-1$
and $\Theta_n(\sigma_n)= \eta_1^{-1} \sigma_1^{-1} \cdots \sigma_{n-2}^{-1}  \sigma_{n-1} \sigma_{n-2} \cdots \sigma_1 \eta_1$.

Using the action by conjugation of $F_n$ on itself we get 
a representation $\chi_n: B_n \ltimes_{\rho_A} F_n \to \Aut F_n$. More precisely,
$\chi_n(\alpha_j)=\rho_A(\sigma_j)$ and $\chi_n(\eta_i)(x_k)= x_i^{-1} x_k x_i$
for any $j=1,\ldots, n-1$ and $i,k=1,\ldots, n$.

 One can easily verify on the images of generators that the composed homomorphism $\chi_n \circ \Theta_n: CB_n  \to \Aut F_n$
coincides with $\rho_{CB} : CB_n  \to \Aut F_n$ defined in Section \ref{sec:necklace}.   We claim  that the kernel of $\chi_n$ 
is generated by $\Theta_n(\zeta^n)$.

Let $w \in \Ker \chi_n$ and write $w$ in the form $w=\alpha  \, \eta$, where $\alpha$ is written in the generators $\alpha_i$'s  and $\eta$ in the generators
$\eta_j$'s; since $\chi_n(w)(x_j)=x_j$ for all generators $x_1, \ldots, x_n$ of $F_n$,
we have that $\chi_n(\alpha)(x_j)=\eta^{-1}(x_j)$, where  $\chi_n(\alpha)(x_j)= \rho_A(\alpha)(x_j)$,
identifying any usual braid generator $\sigma_i$ with the corresponding $\alpha_i$.

It follows that $\rho_A(\alpha)$ is an inner automorphism, therefore $\alpha$ belongs
to the center of $B_n$ (see for instance \cite{BB}, Remark 1): more precisely
$\alpha=((\alpha_{n-1} \cdots \alpha_1)^n)^m$ for some  $m \in \Zz$ and 
$\rho_A(\alpha)(x_j)=(x_1 \cdots x_n)^m x_j (x_1 \cdots x_n)^{-m}$. Then we can deduce 
that $\eta=(\eta_1 \cdots \eta_n)^m$ and $w= ((\alpha_{n-1} \cdots \alpha_1)^n)^m (\eta_1 \cdots \eta_n)^m$.
Using the defining relations
of $B_n \ltimes_{\rho_A} F_n$ we obtain that 
$$w=(((\alpha_{n-1} \cdots \alpha_1) \eta_1)^n)^m=\Theta(\zeta^n)^m \; .$$
Then  $\Ker \chi_n$ 
is generated by $\Theta_n(\zeta^n)$: since $\Theta_n$ is an isomorphism,  
it follows then that the kernel of $\rho_{CB}$ is generated by $\zeta^n$.

\section*{Acknowledgments}

The authors thank the referees for theirs comments and Juan Gonz\'alez-Meneses for useful
discussions on representations of affine braid groups in terms of automorphisms. 
The authors are deeply indebted to Martin Palmer for fruitful conversations on configurations spaces.
The research of the first author was partially supported by the French grant ANR-11-JS01-002-01 "VASKHO".
The research of the second  author was partially supported by the French grant ANR-12-JS01-002-01 "SUSI".


\end{document}